\documentclass[11pt]{article}

\usepackage[margin=1in]{geometry}
\usepackage{amsmath,amssymb,amsthm,mathtools}
\usepackage{paralist}
\usepackage[misc]{ifsym}
\usepackage{epsfig}
\usepackage{epstopdf}
\usepackage[colorlinks=true]{hyperref}
\usepackage{mathrsfs}
\usepackage{dsfont}
\usepackage{enumitem}
\usepackage{tikz}

\hypersetup{
  urlcolor=blue,
  citecolor=red,
  linkcolor=blue
}

\allowdisplaybreaks

\newtheorem{theorem}{Theorem}[section]

\newtheorem{lemma}[theorem]{Lemma}
\newtheorem{proposition}[theorem]{Proposition}

\theoremstyle{definition}
\newtheorem{definition}[theorem]{Definition}
\newtheorem{remark}[theorem]{Remark}
\newtheorem{example}[theorem]{Example}

\makeatletter
\def\namedlabel#1#2{%
  \begingroup
  #2%
  \def\@currentlabel{#2}%
  \phantomsection\label{#1}%
  \endgroup
}
\makeatother

\renewcommand{\H}{\mathcal{H}}
\newcommand{\R}{\mathbb{R}}

\newcommand{\N}{\mathbb{N}}

\renewcommand{\epsilon}{\varepsilon}

\def\B{\mathbb{B}}

\newcommand{\cl}{\operatorname{cl}}
\newcommand{\inte}{\operatorname{int}}
\newcommand{\bd}{\operatorname{bd}}

\DeclareMathOperator{\proj}{proj}

\newcommand{\dnu}{d\nu}

\title{Integral Formulations and a Br\'ezis-Ekeland-Nayroles-Type Principle 
for Prox-Regular Sweeping Processes}

\author{
Juan Guillermo Garrido\footnote{Universidad de Chile, Santiago, Chile.
Email: \texttt{jgarrido@uchile.cl}}
\and
Emilio Vilches\footnote{Universidad de O'Higgins, Rancagua, Chile.
Email: \texttt{emilio.vilches@uoh.cl}}
}

\date{}

\begin{document}

\maketitle

\begin{abstract}
We study sweeping processes in a Hilbert space driven by
time-dependent uniformly prox-regular sets, allowing
bounded-variation discontinuities of the moving constraint.
We introduce a new integral formulation for
bounded-variation trajectories, expressed as a global
variational inequality tested against continuous admissible
trajectories, and compare it with the standard
differential-measure formulation involving the proximal
normal cone. In the prox-regular, generally nonconvex,
setting, this inequality includes a quadratic correction
term arising from the hypomonotonicity of proximal normal
cones.

Under mild assumptions on the moving set, including lower
semicontinuity in time, uniform prox-regularity, and a
selection-extension property ensuring a rich class of test
trajectories, we prove that the integral and
differential-measure formulations are equivalent. This
provides a unified bounded-variation notion of solution for
prox-regular sweeping processes.

We also establish a Br\'ezis--Ekeland--Nayroles-type
variational characterization through a prox-regular
residual: it is nonpositive along admissible trajectories,
and solutions are precisely those for which it vanishes.
As a consequence, we obtain a stability result for uniform
limits of admissible trajectories with vanishing residuals.
\end{abstract}




\section{Introduction}

Let $\H$ be a Hilbert space. The sweeping process is a first-order dynamical system involving the normal cone to a moving closed set. It was introduced by J.-J.~Moreau in the early 1970s as a model for elastoplasticity in contact mechanics (see \cite{MR637727,MR637728,MR508661}).  In the classical (convex) setting, the problem is to find an absolutely continuous trajectory $x\colon [0,T]\to \H$ such that
\begin{equation}\label{SSPP}
\left\{
    \begin{aligned}
    \dot{x}(t)&\in -N(C(t);x(t)) \qquad \textrm{for a.e. } t\in [0,T],\\
    x(0)&=x_0\in C(0).
    \end{aligned}
    \right.
\end{equation}
where $N(C(t);x)$ denotes the (convex) normal cone to $C(t)$ at $x$. 

Because of the presence of the normal cone, the inclusion in \eqref{SSPP} is meaningful only along points $x(t)\in C(t)$. In particular, any solution must satisfy the state constraint $x(t)\in C(t)$ at least for a.e. $t\in [0,T]$. Consequently, existence of solutions in this classical sense requires suitable regularity of the set-valued map $t\rightrightarrows C(t)$. For instance, Moreau's original existence theorem for \eqref{SSPP} (see \cite{MR637727}) assumes Lipschitz continuity in Hausdorff distance:
\begin{align*}
\operatorname{Haus}(C(t),C(s))\leq \kappa \vert t-s\vert  \quad t,s\in [0,T],
\end{align*}
where $\kappa \geq 0$ and $\operatorname{Haus}(A,B)$ denotes the Hausdorff distance between $A$ and $B$.

Motivated by applications in contact mechanics, where collisions, jumps and abrupt changes of the constraint set may occur (see, e.g., \cite{MR3467591,MR1710456,MR2857428}), substantial effort has been devoted to extending the notion of solution to discontinuous moving sets (typically of bounded variation). In this paper, we investigate two notions of solution for discontinuous sweeping processes and prove that they are equivalent under very general assumptions. More precisely, we consider moving sets that are uniformly
prox-regular (hence possibly nonconvex) and merely lower semicontinuous
in time, and we show that the integral formulation
(Definition~\ref{def1}) is equivalent to the standard
differential-measure formulation (Definition~\ref{def2}).  For the convex and bounded case, we refer to \cite[Section~9]{MR4547969}. To the best of our knowledge, an integral variational formulation of this type has not been established in the literature for sweeping processes with uniformly prox-regular (nonconvex) moving sets.

Our analysis also connects to recent work of Krej\v{c}\'{\i}, Monteiro, and
Recupero \cite{MR4484114,MR4643520}, who develop a Kurzweil-integral variational formulation for sweeping processes with uniformly prox-regular constraints in the space
of regulated functions. In the setting of translated constraints
$C(t)=u(t)-Z(w(t))$, they prove well-posedness and identify Moreau's
catching-up scheme as an exact solver for step inputs, leading to
existence by uniform approximation of regulated data \cite{MR4484114}.
More recently, they study vanishing-viscosity regularizations for the
translation-only case $C(t)=u(t)-Z$ and prove that the viscous solutions
converge pointwise, as $\varepsilon\downarrow 0$, to the left-continuous
representative of the Kurzweil solution (with uniform convergence when
the input is continuous), relying on an interior cone-type condition to
obtain uniform variation bounds \cite{MR4643520}. 

The present work is complementary: we work in a Radon-measure BV
framework for general time-dependent uniformly prox-regular sets,
establish the equivalence between the integral and differential-measure
formulations, and develop a variational residual whose vanishing
characterizes solutions. This residual viewpoint yields a stability
principle well suited to approximation (including catching-up and
inexact schemes) and clarifies the link between Kurzweil/Young
formulations and Radon-measure methods.

A central motivation for introducing such a variational residual is the classical
Brezis-Ekeland variational principle and Nayroles' minimum theorem,
which provide a global-in-time variational characterization for a broad
class of evolution problems generated by convex potentials or maximal
monotone operators: one associates to each trajectory a nonnegative
functional whose minimum value is $0$, and solutions are precisely the
trajectories attaining this minimum \cite{MR637215,MR637214,MR418609}.
This perspective has proved particularly effective for existence via
direct methods, as well as for stability and approximation analyses,
including time-discretization schemes \cite{MR2425653,MR2531193,MR2489147}.

In the present paper, however, we use the variational principle
primarily as a characterization and stability tool, rather than as a
direct existence mechanism under the standing assumptions. Indeed, a
direct-method argument would require an additional compactness framework
for admissible trajectories, together with lower semicontinuity of the
associated residual and a recovery construction ensuring that its
infimum is zero. Such ingredients depend on further quantitative
assumptions on the moving set and on the approximation scheme, and are
beyond the scope of the present work.

Motivated by these advantages, in Section~\ref{section5} (Theorem~\ref{thm_BEN})
we prove a Br\'ezis-Ekeland-Nayroles-type principle for sweeping processes
driven by uniformly prox-regular moving sets. In turn, in Section~\ref{section6}
this variational characterization yields a residual-based stability result for
approximations of the moving set: a uniform limit of admissible trajectories
with vanishing residual is a solution of the limit sweeping process.

Taken together, these results build a bridge between a local
characterization, formulated as a measure differential inclusion through
the proximal normal cone, and a global variational characterization,
expressed by an integral inequality tested against admissible
trajectories. This connection unifies the bounded-variation solution
concepts and provides a robust framework for stability and approximation
analyses in the prox-regular, potentially nonconvex, setting.

The paper is organized as follows. In Section~\ref{section2} we fix notation and
collect preliminary material on uniformly prox-regular sets, lower
semicontinuity for set-valued mappings, and measurability tools,
including a Lusin-type result. In Section~\ref{section3} we introduce the geometric
framework for the moving constraint $C(\cdot)$ and present classes of
admissible families, together with sufficient conditions ensuring the
selection-extension property. In Section~\ref{section4} we define integral and
differential-measure solutions and prove their equivalence
(Theorem~\ref{thm_main}). In Section~\ref{section5} we establish a
Br\'ezis-Ekeland-Nayroles-type variational characterization for
prox-regular sweeping processes via the variational residual
(Theorem~\ref{thm_BEN}). Finally, in Section~\ref{section6} we derive a residual-based
stability result for approximations of the moving set: a uniform limit of
admissible trajectories with vanishing residual is a solution of the
limit sweeping process.

	\section{Notation and Preliminary Results}\label{section2}

    Throughout the paper, $\H$ stands for a Hilbert space endowed with the inner product $\langle\cdot,\cdot\rangle$ and norm $\|\cdot\|$. The closed (resp. open) ball centered at $x$ with radius $r>0$ is denoted by $\mathbb{B}_r[x]$ (resp. $\mathbb{B}_r(x)$), and the closed unit ball is denoted by $\mathbb{B}$. For a subset $A\subset\H$, we denote the closure, interior and boundary of $A$, respectively, as $\cl A$, $\inte A$, and $\bd A$. Given two points $x,y\in \H$, we define the interval $$[x,y]_{\H} := \{tx+(1-t)y : t\in [0,1]\}.$$
Given $\rho\in ]0,+\infty]$ and $\gamma\in ]0,1[$, the  \emph{$\rho$-enlargement} and the 
	\emph{$\gamma\rho$-enlargement} of $S$ are 
	$$
	U_\rho(S) := \{x\in \H:d(x;S)<\rho\} \textrm{ and } U_\rho^\gamma(S):=U_{\gamma\rho}(S),
	$$
    where $d(x;S):=\inf_{y\in S}\|x-y\|$ denotes the distance from $x$ to $S$.
    The \emph{projection} of $x$ onto $S\subset \H$ is the (possibly empty) set  
$$\operatorname{Proj}_{S}(x):=\left\{z\in S : d_{S}(x)=\Vert x-z\Vert\right\}.$$ When the above set is a singleton, we denote its unique element by $\operatorname{proj}_{S}(x)$.

\subsection{Elements of Nonsmooth Analysis}
    
 Let $f\colon \H\to \mathbb{R}\cup\{+\infty\}$ be a \emph{lower semicontinuous} (lsc) function and let $x\in \operatorname{dom}f$. We say that $\zeta$ belongs to the \emph{proximal subdifferential}  of $f$ at $x$, denoted  $\partial_P f(x)$, if there exist $\sigma\geq 0$ and $\eta\geq 0$ such that
		\begin{equation*}
			f(y)\geq f(x)+\left\langle \zeta,y-x\right\rangle -\sigma\Vert y-x\Vert^2 \textrm{ for all } y\in \mathbb{B}_{\eta}(x).
		\end{equation*}
	Using this notion, the \emph{proximal normal cone} of a closed set $S\subset \H$ at $x\in S$ is
	$$
	N^P(S;x):=\partial_P \delta_S(x),
	$$
	where $\delta_S$ is the indicator function of $S\subset \H$ (i.e., $\delta_S(x)=0$ if $x\in S$ and $\delta_S(x)=+\infty$ otherwise). Moreover, it is well known that if $S$ is convex, then the proximal normal cone
coincides with the (convex) normal cone
\[
N^{\operatorname{conv}}(S;x):=\bigl\{\zeta\in\H:\ \langle \zeta,\,y-x\rangle \le 0
\ \text{for all } y\in S\bigr\}.
\]
We refer to \cite{MR1488695} for further details.

 \subsection{Prox-regular Sets}   
	We recall the notion of uniformly prox-regular sets. In finite dimension,
it goes back to Federer \cite{MR110078} under the name of \emph{sets of positive
reach}, and it was later developed in a variational-analytic framework by
Rockafellar, Poliquin, and Thibault \cite{MR1694378}. Uniform prox-regularity
provides a geometric setting that strictly extends convexity while covering
broad classes of nonconvex sets with sufficiently smooth geometry (for instance,
$C^{2}$-smooth embedded submanifolds and $C^{2}$-domains). For further background
and surveys, see \cite{MR2768810,MR4659163}.
	
\begin{definition}
Let $S\subset \H$ be nonempty and closed and $\rho\in ]0,+\infty]$. We say that $S$ is \emph{$\rho$-uniformly prox-regular} if, for all $x\in S$ and $\zeta\in N^P(S;x)$, 
		\begin{equation*}
			\langle \zeta,x'-x\rangle\leq \frac{\|\zeta\|}{2\rho}\|x'-x\|^2 \textrm{ for all } x'\in S.
		\end{equation*}
	\end{definition}
	Note that every closed convex set is $\rho$-uniformly prox-regular for every $\rho>0$. The following proposition provides a characterization of uniformly prox-regular sets in Hilbert spaces (see, e.g.,  \cite{MR2768810}).
	\begin{proposition}\label{prox_reg_prop}
		Let $S\subset \H$ be a closed set and $\rho\in ]0,+\infty]$. The following assertions are equivalent:
		\begin{enumerate}
			\item [(a)] $S$ is $\rho$-uniformly prox-regular.
			\item [(b)] For any positive $\gamma<1$ the mapping $\proj_S$ is well-defined on $U_\rho^\gamma(S)$ and Lipschitz continuous on $U_\rho^\gamma(S)$ with $(1-\gamma)^{-1}$ as a Lipschitz constant, i.e.,
			$$\left\|\proj_S\left(u_1\right)-\proj_S\left(u_2\right)\right\| \leq(1-\gamma)^{-1}\left\|u_1-u_2\right\| \quad$$ for all $u_1, u_2 \in U_\rho^\gamma(S)$.
		\end{enumerate}
	\end{proposition}

	We establish the following new property of uniformly prox-regular sets. A related result can be found in \cite[Lemma 5 \& Lemma 6]{MR2406391}.
	\begin{lemma}\label{lemma-prox-inter-ball}
		Let $S$ be a closed and $\rho$-uniformly prox-regular set, and let $a\in \H$ and $r<\rho$ be such that $S\cap \mathbb{B}_r(a)\neq \emptyset$. Then $\cl{(S\cap \mathbb{B}_r(a))} = S\cap \mathbb{B}_r[a]$, and the set $S\cap \mathbb{B}_r[a]$ is $\rho$-uniformly prox-regular.
	\end{lemma}
	\begin{proof}
    Take $x\in S\cap \mathbb{B}_r[a]$. The nontrivial case is when $x\in \bd S$ and $\|x-a\| = r$. Define $u_t := ta+(1-t)x$ for every $t\in ]0,1]$. Then $\mathbb{B}_{t\|a-x\|}(u_t)\subset \mathbb{B}_{r}(a)$. Suppose that there exists $t\in ]0,1]$ such that  $S\cap \mathbb{B}_{t\|a-x\|}(u_t) = \emptyset$. It then follows that $a-x\in N^P(S;x)$. Since $S$ is $\rho$-uniformly prox-regular, we deduce that  $\mathbb{B}_r(a)\cap S = \emptyset$ (because $r<\rho$), which contradicts our assumption. Therefore, for all $t\in ]0,1]$, we have $S\cap \mathbb{B}_{t\|a-x\|}(u_t) \neq \emptyset$. Hence, there exists a sequence $(y_n)$ with $y_n\in S\cap \mathbb{B}_{t_n\|a-x\|}(u_{t_n})$ for some sequence $(t_n)\searrow 0$, and $y_n\to x$ as $n\to\infty$. Consequently, $x\in \cl{(S\cap \mathbb{B}_r(a))}$. The $\rho$-uniform prox-regularity of $S\cap \mathbb{B}_r[a]$ follows from \cite[Corollary 16.17]{MR4659163}.
	\end{proof}
    
The following proposition establishes a new topological property of uniformly prox-regular sets in Hilbert spaces: connectedness and path connectedness are equivalent.
    \begin{proposition}\label{path-connected-prox-reg}
    Let $S\subset \H$ be a closed and $\rho$-uniformly prox-regular set. Then $S$ is connected if and only if it is path-connected.  
    \end{proposition}
    \begin{proof}
    Every path-connected set is connected. Suppose that $S$ is connected.    Fix $\gamma\in ]0,1[$. We first claim that $U_\rho^\gamma(S)$ is connected. Assume by contradiction that $U_\rho^\gamma(S)$ is not connected. Then there exists a nonconstant continuous map $g\colon U_\rho^\gamma(S)\to \{0,1\}$. Hence there exist $x_0,x_1\in U_\rho^\gamma(S)$ such that $g(x_0) = 0$ and $g(x_1) = 1$. Since $S$ is $\rho$-uniformly prox-regular and $U_\rho^\gamma(S)\subset U_{\rho}(S)$, the metric projection onto $S$ is single-valued on $U_\rho^\gamma(S)$. Define $y_i := \proj_{S}(x_i)$ for $i\in\{0,1\}$. Then $y_i\in S$.  Observe that $[x_i,y_i]_{\H}\subset U_\rho^\gamma(S)$ for $i\in \{0,1\}$. Since each segment is connected and $g$ is continuous, $g$ is constant on $[x_i,y_i]_{\H}$. Therefore $g(y_0) = 0$ and $g(y_1) = 1$. However, $y_0, y_1\in S$, which contradicts the connectedness of  $S$, because the restriction of $g$ to $S$ must be constant.  On the other hand, since $U_\rho^\gamma(S)$ is open and connected, it is path-connected (see, e.g., \cite[Theorem V.5.5]{MR478089}). Finally, note that $\proj_S\colon U_\rho^\gamma(S)\to S$ is continuous (see Proposition \ref{prox_reg_prop}) and that $\proj_S(U_\rho^\gamma(S)) = S$. Since the continuous image of a path-connected set is path-connected (see, e.g., \cite[p. 115]{MR478089}), we conclude that $S$ is path-connected.
    \end{proof}

\subsection{Set-Valued Maps and Lusin’s Theorem}

 Let $X,Y$ be topological spaces. We recall that a set-valued mapping $\mathscr{C}\colon X \rightrightarrows  Y$ is lower semicontinuous (lsc for short) at $\bar x\in X$ if for every open set $U\subset Y$ such that $\mathscr{C}(\bar x)\cap U\neq \emptyset$, the set $\{ x\in X : \mathscr{C}(x)\cap U\neq \emptyset \}$ is a neighborhood of $\bar x$.
    
Lusin's theorem will be a useful tool in our analysis. We state below a version for nonseparable Hilbert spaces. 
    \begin{proposition}
        Let $(X,d)$ be a complete and separable metric space and $\nu$ be a finite measure defined on the Borel sets of $X$. Suppose that $f\colon X\to \H$ is a strongly measurable function. Then, for every $\epsilon>0$, there exists a compact set $K\subset X$ such that $\nu(X\setminus K)<\epsilon$ and $f|_{K}$ is continuous.
    \end{proposition}
    \begin{proof}
        This follows directly from Pettis's Theorem (see, e.g., \cite[Theorem II.2]{MR0453964}) together with \cite[Theorem 7.1.13]{MR2267655}.
    \end{proof}

    \begin{lemma}\label{patching-lemma}
Let $C\colon [0,T] \rightrightarrows  \H$ be lower semicontinuous, let $K\subset [0,T]$ be compact, and let
$y\colon K\to \H$ be continuous with $y(t)\in C(t)$ for all $t\in K$. Define
\[
D(t):=
\begin{cases}
\{y(t)\}, & t\in K,\\
C(t), & t\in [0,T]\setminus K.
\end{cases}
\]
Then $D$ is lower semicontinuous on $[0,T]$.
\end{lemma}
\begin{proof}
Let $x\in \H$ and define $\phi(t) := \operatorname{dist}(x;D(t))$ for $t\in [0,T]$. By \cite[Proposition 6.1.15 (a)]{MR2527754},  it suffices to prove that $\phi$ is upper semicontinuous. Fix $\bar t\in [0,T]$. If $\bar t\notin K$, then, since $K$ is closed, we have $\operatorname{dist}(\bar t;K)>0$. In particular, there exists $\delta>0$ such that $]\bar t-\delta,\bar t+\delta[\cap K =\emptyset$, and therefore $D( t)=C(t)$ for all $t\in [0,T]\cap ]\bar t-\delta,\bar t+\delta[$. Since $C$ is lower semicontinuous, \cite[Proposition 6.1.15 (a)]{MR2527754} implies that $t\mapsto \phi(t)$ is upper semicontinuous at $\bar t$. Assume now that $\bar t\in K$, and let $(t_n)$ be any sequence with $t_n\to \bar t$. Passing to a subsequence $(t_{n_k})$, we may assume that  $$\limsup_{n\to \infty} \phi(t_n) = \lim_{k\to \infty} \phi(t_{n_k}).$$ 
    Set 
    $$J_1 := \{ k\in \N : t_{n_k}\in K \}, \qquad  J_2 := \{k\in \N : t_{n_k}\notin K\}.$$  
    At least one of $J_1$ or $J_2$ is infinite. Moreover, since $(\phi(t_{n_k}))$ converges, every subsequence has the same limit.\\
    If $J_1$ is infinite, then for $k\in J_1$ we have $D(t_{n_k})=\{y(t_{n_k})\}$, hence
    \begin{equation*}
        \begin{aligned}
        \lim_{k\to \infty} \phi(t_{n_k}) &= \lim_{k\to \infty, k\in J_1} \|x-y(t_{n_k})\|=\|x-y(\bar t)\| = \phi(\bar t),
        \end{aligned}
    \end{equation*}
    where we used the continuity of $y$ on $K$. If $J_2$ is infinite, then for $k\in J_2$ we have $D(t_{n_k})=C(t_{n_k})$, and therefore
    \begin{equation*}
        \begin{aligned}
        \lim_{k\to \infty} \phi(t_{n_k}) &= \lim_{k\to \infty, k\in J_2} \operatorname{dist}(x;C(t_{n_k}))\leq \operatorname{dist}(x;C(\bar t)) \leq \|x-y(\bar t)\|=\phi(\bar t).
        \end{aligned}
    \end{equation*}
    Here we used that $t\mapsto \operatorname{dist}(x;C(t))$ is upper semicontinuous and that $y(\bar t)\in C(\bar t)$. In either case, we obtain 
    $$
    \limsup_{n\to\infty}\phi(t_n)=\lim_{k\to \infty}\phi(t_{n_k})\le \phi(\bar t),
    $$
    which proves that $\phi$ is upper semicontinuous at $\bar t$. This concludes the proof.
\end{proof}

\subsection{Radon measures}

Throughout, $\nu$ denotes a \emph{positive Radon measure} on $[0,T]$, that is,
a Borel measure on $[0,T]$ which is finite on compact sets and inner regular:
\[
\nu(B)=\sup\{\nu(K): K\subset B,\ K \textrm{ compact}\}
\qquad\textrm{for every Borel set }B\subset[0,T].
\]
Since $[0,T]$ is compact, every positive Radon measure $\nu$ on $[0,T]$ is
finite. When we refer to a complete measure $\nu$, it means that $\nu$ is defined in an  extension of the Borelians of $[0,T]$, let us say a $\sigma$-algebra $\mathcal{A}$ such that the measure space $([0,T],\mathcal{A},\nu)$ is complete.

Let $1\le p<\infty$. The Bochner space $L^p_\nu([0,T];\H)$ consists of (strongly)
$\nu$-measurable mappings $w\colon[0,T]\to\H$ such that
\[
\int_{0}^{T}\|w(t)\|^{p}\,d\nu(t)<+\infty,
\]
where mappings are identified if they agree $\nu$-a.e. It is a Banach space
for the norm
\[
\|w\|_{L^p_\nu}:=
\Bigl(\int_{0}^{T}\|w(t)\|^{p}\,d\nu(t)\Bigr)^{1/p}.
\]
Moreover, $L^\infty_\nu([0,T];\H)$ denotes the space of functions $w\colon[0,T]\to\H$ which are
strongly $\nu$-measurable and $\|w\|_{L^\infty_\nu}:=\operatorname*{ess\,sup}_{t\in[0,T]}\|w(t)\|<+\infty$. We refer to \cite{MR2267655} for more details.
\subsection{Differential measures}

Solutions to measure differential inclusions are typically sought among
trajectories of bounded variation (see, e.g., \cite{MR3274844,MR3716922,MR4950495,MR637728}). We recall the standard definition on $[0,T]$. Let $T>0$ and let $x\colon[0,T]\to\H$ be a mapping. A \emph{subdivision} of
$[0,T]$ is a finite sequence
\[
\sigma=(t_0,\dots,t_k)\in\R^{k+1},
\qquad k\in\N,
\qquad 0=t_0<t_1<\cdots<t_k=T.
\]
To such a subdivision we associate
\[
S_\sigma(x):=\sum_{i=1}^k \bigl\|x(t_i)-x(t_{i-1})\bigr\|.
\]
The \emph{total variation} of $x$ on $[0,T]$ is the extended real number
\[
\operatorname{var}(x;[0,T]):=\sup_{\sigma\in\mathcal S} S_\sigma(x),
\]
where $\mathcal S$ denotes the set of all subdivisions of $[0,T]$. We say that
$x$ has \emph{bounded variation} on $[0,T]$ (briefly,
$x\in BV([0,T];\H)$) if $\operatorname{var}(x;[0,T])<+\infty$.  For further background on functions of bounded variation in Banach spaces, we
refer to \cite{MR206190,MR512194}.

Assume in addition that $x(\cdot)$ is \emph{right-continuous} on $[0,T]$ and of
bounded variation. Then one can associate with $x$ an $\H$-valued (countably
additive) vector measure $dx$ on $[0,T]$ (see, e.g., \cite{MR206190,MR512194})
such that, for all $0\le s\le t\le T$,
\[
x(t)=x(s)+\int_{]s,t]} dx.
\]
The measure $dx$ is called the \emph{differential measure} (or \emph{Stieltjes
measure}) of $x$. Conversely, let $\nu$ be a positive Radon measure on $[0,T]$, let
$x\colon[0,T]\to\H$ be a mapping, and let $v\in L^1_\nu([0,T];\H)$. Fix
$t_0\in[0,T]$ and assume that
\[
x(t)=x(t_0)+\int_{]t_0,t]} v(s)\,d\nu(s)
\qquad \textrm{ for all } t\in[0,T].
\]
Then $x(\cdot)$ is right-continuous on $[0,T]$, has bounded variation, and its
differential measure satisfies $dx = v\,d\nu$. In particular, $v$ is a Radon-Nikod\'ym density of $dx$ with respect to $\nu$.

	\section{Geometric Framework and Examples of Admissible Moving Sets}\label{section3}

In this section we state the geometric assumptions on the moving constraint set \(C(\cdot)\) used throughout the paper and present classes of admissible families for which these hypotheses can be verified. We require: uniform prox-regularity of the values \(C(t)\), lower semicontinuity in time, and a bounded selection-extension property along trajectories. The latter ensures that continuous selections prescribed on compact time sets and remaining in a fixed tube around a bounded admissible trajectory can be extended to global continuous selections with a uniform bound. This guarantees the availability of sufficiently rich test trajectories and will be used repeatedly in the variational arguments developed later.

Let $\H$ be a Hilbert space, and consider a set-valued mapping $C\colon [0,T] \rightrightarrows  \H$ satisfying the following assumptions:
	\begin{enumerate}
		\item[\namedlabel{H1}{$(\mathsf{H}_1)$}] For all $t\in [0,T]$, the set $C(t)$ is nonempty, closed, connected, and $\rho$-uniformly prox-regular.
        \item[\namedlabel{H2}{$(\mathsf{H}_2)$}] The mapping $C$ is lower semicontinuous on $[0,T]$.
		\item[\namedlabel{H3}{$(\mathsf{H}_3)$}] For every bounded mapping $x\colon [0,T]\to \H$ such that $x(t)\in C(t)$ for all $t\in [0,T]$ and all $\eta>0$, there exists a constant $R_{x,\eta}>0$ such that the following holds: for every compact set $K\subset [0,T]$, every continuous selection $y\colon K\to \H$ of $C$ with $\sup_{t\in K}\|y(t)-x(t)\|\leq \eta$, there exists a continuous selection $\bar y\colon [0,T]\to \H$ of $C$ such that
\[
\bar y=y \textrm{ on } K  \textrm{ and } \sup_{t\in [0,T]}\|\bar y(t)\|\le R_{x,\eta}.
\]
\end{enumerate}

The next proposition provides convenient sufficient conditions ensuring that
the bounded selection extension property along trajectories \ref{H3} holds.
In particular, \ref{H3} is satisfied for uniformly prox-regular moving sets
with uniformly bounded values, and also in the convex case, where no a priori
uniform boundedness of $t\mapsto C(t)$ is required.
\begin{proposition}
    Assume that \ref{H2} holds. Then each of the following additional assumptions implies \ref{H3}:
    \begin{itemize}
        \item[(i)]  Assumption \ref{H1} holds and there exists $r>0$ such that $\bigcup_{t\in [0,T]} C(t)\subset r\mathbb{B}$.
        \item[(ii)] For every $t\in [0,T]$, the set $C(t)$ is nonempty, closed and convex.
    \end{itemize}
\end{proposition}

\begin{proof} Fix a bounded mapping \(x\colon [0,T]\to \H\) such that \(x(t)\in C(t)\) for all \(t\in [0,T]\) and $\eta>0$. Fix a compact set $K\subset [0,T]$, and a continuous selection $y\colon K\to \H$ of $C$ such that $\sup_{t\in K}\|y(t)-x(t)\|\leq \eta$. Define
$$
D(t):=\begin{cases}
    \{y(t)\}, & \textrm{if } t\in K,\\
    C(t), & \textrm{if } t\in [0,T]\setminus K.
\end{cases}
$$
By Lemma~\ref{patching-lemma} and \ref{H2}, the set-valued mapping \(D\) is lower semicontinuous. Moreover, for every \(t\in [0,T]\), the set \(D(t)\) is nonempty, closed, \(\rho\)-uniformly prox-regular, and connected. Hence, by Proposition \ref{path-connected-prox-reg}, each value $D(t)$ is path-connected. Therefore, by  \cite[Theorem~5]{MR2406391}, there exists a continuous selection $\bar{y}\colon [0,T]\to \H$ such that $\bar{y}(t)\in D(t)$ for all $t\in [0,T]$. 
In particular, $\bar{y}=y$ on $K$. Assume that assumption (i) holds, then there exists $r>0$ such that $C(t)\subset r\mathbb{B}$ for all $t\in [0,T]$. Since $\bar{y}(t)\in D(t)\subset C(t)$ for all $t\in [0,T]$, we obtain that $\sup_{t\in [0,T]}\Vert \bar{y}(t)\Vert \leq r$. Thus, \ref{H3} holds with $R_{x,\eta}:=r$.\\
To prove \ref{H3} under assumption~(ii), set $M:=\sup_{t\in [0,T]}\Vert x(t)\Vert$. Then, for every $t\in K$,
$$
\Vert y(t)\Vert \leq \Vert y(t)-x(t)\Vert +\Vert x(t)\Vert \leq \eta +\sup_{t\in [0,T]}\Vert x(t)\Vert=M.
$$
By Michael's selection theorem (see, e.g., \cite[Theorem 6.3.6]{MR2527754}), there exists a continuous selection $z\colon [0,T] \to \H$ such that $z(t)\in C(t)$ for all $t\in [0,T]$. Let $R_0:=\sup_{t\in [0,T]}\Vert z(t)\Vert$, choose $R_{x,\eta}>\max\{M,R_0\}$, and define
$$
\Phi(t):=C(t)\cap \mathbb{B}_{R_{x,\eta}}(0), \qquad t\in [0,T].
$$
Then, by \cite[Proposition 6.1.24]{MR2527754}, $\Phi$ is lower semicontinuous and has nonempty and convex values. Define
$$
\tilde{D}(t):=\begin{cases}
    \{y(t)\}, & \textrm{if } t\in K,\\
    \Phi(t), & \textrm{if } t\in [0,T]\setminus K.
\end{cases}
$$
By Lemma \ref{patching-lemma}, the set-valued map $\tilde{D}(\cdot)$ is lower semicontinuous, thus we take $F := \cl \tilde{D}$, by \cite[Proposition 6.1.19]{MR2527754} we have $F$ is lower semicontinuous, and it takes nonempty, convex and closed values. Applying Michael’s selection theorem, there exists a continuous selection $\tilde{y}\colon [0,T]\to \H$ such that $\tilde{y}(t)\in F(t)$ for all $t\in [0,T]$. In particular, $\tilde{y}=y$ on $K$. Moreover, if $t\in K$, then $\Vert \tilde{y}(t)\Vert=\Vert y(t)\Vert\leq M<R_{x,\eta}$, while if $t\in [0,T]\setminus K$, then $\tilde{y}(t)\in F(t)\subset \mathbb{B}_{R_{x,\eta}}[0]$, so $\Vert \tilde{y}(t)\Vert \leq R_{x,\eta}$.  Hence
$$
\sup_{t\in [0,T]}\Vert \tilde{y}(t)\Vert \leq R_{x,\eta}.
$$
Thus, \ref{H3} holds.
\end{proof}

The next example presents a class of moving sets for which property \ref{H3} can be verified explicitly.
\begin{example}
    Let $\psi\colon [0,T]\times \H\to \R$ be continuous and assume that it
maps bounded sets into bounded sets. Define $C(t):=\operatorname{epi}\psi(t,\cdot)\subset \H\times\R$.
Then $C(\cdot)$ satisfies \ref{H3}. Indeed, consider a bounded function $(x,\lambda)\colon [0,T]\to \H\times \R$ such that $\psi(t,x(t))\leq \lambda(t)$ and $\eta>0$. Fix a compact set $K\subset [0,T]$ and a continuous selection $(y,\alpha)$ of $C$ on $K$ such that $\|x-y\|_\infty\leq \eta$ and $\|\lambda-\alpha\|_\infty\leq \eta$. Define 
    \begin{equation*}
        M:= \max\{\|x\|_\infty,\|\lambda\|_\infty\}<\infty.
    \end{equation*}
    Note that $\max\{\|y\|_\infty,\|\lambda\|_\infty\}\leq M+\eta$. Consider a continuous extension of $y$, given by $\hat y\colon [0,T]\to \H$ such that $\|\hat y\|_\infty\leq M+\eta=:R$. Such selection exist since we can consider 
    \begin{equation*}
        D(t) = \begin{cases}
            \{y(t)\} : t\in K\\
            \mathbb{B}_{R}[0] : t\notin K
        \end{cases}
    \end{equation*}
    which is a lower semicontinuous set-valued mapping (by Proposition \ref{patching-lemma}) and takes convex, closed and nonempty values. Therefore, by Michael's selection theorem, $D$ has continuous selections. Take $\hat{R}>0$ such that $\psi([0,T]\times\B_{R}(0))\subset [-\hat R,\hat R]$. Finally, seeing that $\psi(t,\hat{y}(t))\leq \hat{R}$ for all $t\in [0,T]$, we can take $R_{x,\eta} := \hat{R} + R$. 
\end{example}
It is also natural to expect \ref{H3} to be stable under transformations
that transport the geometry of the constraints. In particular, starting
from any family that satisfies \ref{H3}, applying a smooth change of
variables whose forward and inverse maps preserve boundedness should
yield a new family that still satisfies \ref{H3}, since any admissible
selection on a compact time set can be pulled back, extended using
\ref{H3} for the original family, and then pushed forward. We do not
formalize this invariance principle here, as it would require
introducing additional assumptions and notation that are not needed for
the main results of this paper.

We end this section with an example of an admissible geometry for which assumption~\ref{H3} fails.
\begin{example}\label{ex:H3-fails}
Assumptions \ref{H1} and \ref{H2} do not imply \ref{H3}.  Let $\H=\R^2$ and
$T=1$. For $t\in[0,1]\setminus\{\tfrac12\}$ set $h(t):=|t-\tfrac12|^{-1}$
(so that $h(t)\ge 2$), and define
\[
C(t):=L_-\cup\Gamma^-\cup S_-(t)\cup B(t)\cup S_+(t)\cup\Gamma^+\cup L_+,
\qquad
C(\tfrac12):=L_-,
\]
where
\[
\begin{aligned}
L_-&:=\,]-\infty,-\tfrac32]\times\{0\},
&
L_+&:=[\tfrac32,+\infty[\,\times\{0\},\\[2pt]
\Gamma^-&:=\bigl\{(-\tfrac32,-\tfrac12)+\tfrac12(\cos\theta,\sin\theta):
\theta\in[0,\tfrac{\pi}{2}]\bigr\},
&
\Gamma^+&:=\bigl\{(\tfrac32,-\tfrac12)+\tfrac12(\cos\theta,\sin\theta):
\theta\in[\tfrac{\pi}{2},\pi]\bigr\},\\[2pt]
S_\pm(t)&:=\{\pm1\}\times[-h(t),-\tfrac12],
&
B(t)&:=\bigl\{(0,-h(t))+(\cos\theta,\sin\theta):\theta\in[\pi,2\pi]\bigr\}.
\end{aligned}
\]
Thus, for $t\neq\tfrac12$, the set $C(t)$ consists of the two horizontal
half-lines $L_\pm$, joined by a $\mathsf U$-shaped bridge of depth $h(t)$:
two vertical strands $S_\pm(t)$, the lower semicircle $B(t)$ of radius $1$,
and two circular fillets $\Gamma^\pm$ of radius $\tfrac12$ making the
junction of class $C^{1,1}$ (see Figure~\ref{fig:H3-fails}). \begin{figure}[htbp]
\centering
\begin{tikzpicture}[scale=0.8, line cap=round, line join=round]
  \draw[gray!50, thin] (-4.5,0) -- (4.5,0);
  \draw[very thick, blue!65!black]
    (-4.3,0) -- (-1.5,0)                           
    arc[start angle=90, end angle=0, radius=0.5]   
    -- (-1,-3)                                     
    arc[start angle=180, end angle=360, radius=1]  
    -- (1,-0.5)                                    
    arc[start angle=180, end angle=90, radius=0.5] 
    -- (4.3,0);                                    
  \node[above, blue!65!black] at (-3.2,0.02) {\footnotesize $L_-$};
  \node[above, blue!65!black] at (3.2,0.02) {\footnotesize $L_+$};
  \node[left,  blue!65!black] at (-1.05,-1.8) {\footnotesize $S_-(t)$};
  \node[right, blue!65!black] at (1.05,-1.8) {\footnotesize $S_+(t)$};
  \node[blue!65!black] at (0,-3.35) {\footnotesize $B(t)$};
  \draw[dashed, gray] (0,0) -- (0,-2.85);
  \draw[<->, gray] (0.12,-0.12) -- (0.12,-2.85);
  \node[right, gray] at (0.07,-1.45) {\footnotesize $h(t)$};
  \fill[black] (0,-4) circle (1.4pt);
  \node[below] at (0,-4.05) {\footnotesize $(0,-h(t)-1)$};
  \draw[->, thick, red!65!black] (2.1,-3.2) -- (2.1,-4.5);
  \node[right, red!65!black, align=left] at (2.2,-3.85)
    {\footnotesize the bridge escapes\\[-2pt]\footnotesize as $t\to\tfrac12$};
  \fill[black] (-2,0) circle (1.6pt);
  \node[above] at (-2,0.05) {\footnotesize $(-2,0)$};
  \fill[black] (2,0) circle (1.6pt);
  \node[above] at (2,0.05) {\footnotesize $(2,0)$};
  \node at (0,1) {\small $C(t)$, \ $t\neq\tfrac12$};
\end{tikzpicture}
\qquad
\begin{tikzpicture}[scale=0.8]
  \draw[gray!50, thin] (-4.5,0) -- (1.2,0);
  \draw[very thick, blue!65!black] (-4.3,0) -- (-1.5,0);
  \fill[blue!65!black] (-1.5,0) circle (1.6pt);
  \node[below] at (-1.5,-0.08) {\footnotesize $(-\tfrac32,0)$};
  \node[above, blue!65!black] at (-3.2,0.02) {\footnotesize $L_-$};
  \fill[black] (-2,0) circle (1.6pt);
  \node[above] at (-2,0.08) {\footnotesize $(-2,0)$};
  \path (0,-4.6) -- (0,1);
  \node at (-1.6,1) {\small $C(\tfrac12)=L_-$};
\end{tikzpicture}
\caption{The moving set of Example~\ref{ex:H3-fails}. Left: for
$t\neq\tfrac12$, the half-lines $L_\pm$ are joined by a $C^{1,1}$ bridge
of depth $h(t)=|t-\tfrac12|^{-1}$. Right: at $t=\tfrac12$ only $L_-$
remains. Any continuous selection passing from $(2,0)$ to $L_-$ near
$t=\tfrac12$ must traverse the bottom of the bridge, forcing its
uniform norm to exceed $1/\delta$; hence \ref{H3} fails while
\ref{H1} and \ref{H2} hold.}
\label{fig:H3-fails}
\end{figure}
As
$t\to\tfrac12$, the bridge escapes to infinity, and at $t=\tfrac12$ only
the left half-line survives. Note that $L_-\subset C(t)$ and
$L_+\subset C(t)$ for all $t\neq\tfrac12$, while $L_-\subset C(t)$ for
\emph{all} $t\in[0,1]$.\\
\emph{Verification of \ref{H1}.} For $t\neq\tfrac12$, the set $C(t)$ is a
closed, connected, complete embedded curve of class $C^{1,1}$, whose
curvature is bounded by $2$  and whose distinct strands are at mutual distance at least
$2$ (the two vertical strands), the remaining pairs of pieces being at
distance at least $\tfrac32$ since $h(t)\geq 2$. It is readily verified that $C(t)$ is uniformly $\rho$-prox-regular, with $\rho=\tfrac12$. The set
$C(\tfrac12)$ is closed and convex, hence $\rho$-uniformly prox-regular
for every $\rho>0$.\\
\emph{Verification of \ref{H2}.} For $s,t$ in a compact subinterval of
$[0,1]\setminus\{\tfrac12\}$, the sets $C(t)$ and $C(s)$ differ only by
the depth of the bridge, and a direct computation gives
$\operatorname{Haus}(C(t),C(s))\le |h(t)-h(s)|$; hence $C(\cdot)$ is
(Hausdorff) continuous, in particular lower semicontinuous, on
$[0,1]\setminus\{\tfrac12\}$. At $\bar t=\tfrac12$, lower semicontinuity
holds because $C(\tfrac12)=L_-\subset C(t)$ for every $t\in[0,1]$: any
open set meeting $C(\tfrac12)$ meets $C(t)$ for all $t$. Observe that
lower semicontinuity is precisely the right one-sided notion here: the
sudden collapse of the moving set at $t=\tfrac12$ is compatible with
\ref{H2} but not with Hausdorff continuity.\\
\emph{Failure of \ref{H3}.} Consider the bounded admissible mapping
$x(t):\equiv(-2,0)\in L_-\subset C(t)$ for all $t\in[0,1]$, and let
$\eta:=4$. For $\delta\in\,]0,\tfrac14[$ define the compact set
$K_\delta:=[0,\tfrac12-\delta]\cup[\tfrac12+\delta,1]$ and the mapping
\[
y_\delta(t):=
\begin{cases}
(-2,0), & t\in[0,\tfrac12-\delta],\\
(2,0), & t\in[\tfrac12+\delta,1],
\end{cases}
\]
which is a continuous selection of $C$ on $K_\delta$ (both pieces of
$K_\delta$ are relatively open and closed) satisfying
$\sup_{t\in K_\delta}\|y_\delta(t)-x(t)\|=4=\eta$. Let
$\bar y\colon[0,1]\to\R^2$ be \emph{any} continuous selection of $C$
with $\bar y=y_\delta$ on $K_\delta$, and write
$\bar y=(\bar y_1,\bar y_2)$. Since
$\bar y(\tfrac12)\in C(\tfrac12)=L_-$, we have
$\bar y_1(\tfrac12)\le -\tfrac32$, while
$\bar y_1(\tfrac12+\delta)=2$. By the intermediate value theorem there
exists $\tau\in\,]\tfrac12,\tfrac12+\delta]$ such that
$\bar y_1(\tau)=0$. On the other hand,
$C(\tau)\cap(\{0\}\times\R)=\{(0,-h(\tau)-1)\}$, the lowest point of the
bridge. Therefore
\[
\sup_{t\in[0,1]}\|\bar y(t)\|\ \ge\ \|\bar y(\tau)\|
\ =\ h(\tau)+1\ >\ \frac{1}{\delta}.
\]
Since $\delta\in\,]0,\tfrac14[$ is arbitrary, no constant $R_{x,\eta}$
can satisfy the requirement of \ref{H3}: the property fails. We
emphasize that, for each \emph{fixed} $\delta$, global continuous
extensions of $y_\delta$ do exist: the set-valued mapping equal to
$\{y_\delta(t)\}$ on $K_\delta$ and to $C(t)$ elsewhere is lower
semicontinuous by Lemma~\ref{patching-lemma}, takes
$\rho$-uniformly prox-regular and path-connected values by
Proposition~\ref{path-connected-prox-reg}, and hence admits a continuous
selection by \cite[Theorem~5]{MR2406391}. Thus, under \ref{H1} and
\ref{H2}, the extension itself is always available, and the genuine
content of \ref{H3} is the \emph{uniform bound} $R_{x,\eta}$,
independent of the compact set $K$ and of the partial selection $y$.
\end{example}

\section{Equivalence Between Differential-Measure and Integral Solutions}\label{section4}

In this section, we establish the equivalence between two bounded-variation
solution concepts for the sweeping process. 
On the one hand, we consider a \emph{local} formulation in terms of the
differential measure $dx$ of the trajectory, requiring that its density
(with respect to a suitable reference measure) belongs to the negative
proximal normal cone.
On the other hand, we consider a \emph{global} variational inequality tested
against admissible trajectories.
In the prox-regular setting, this inequality necessarily involves a
quadratic correction term reflecting the hypomonotonicity of proximal normal
cones, while in the convex case the correction is not needed.

A key technical point is the availability of sufficiently rich test
trajectories.
This is ensured by the bounded selection extension property \ref{H3},
which allows one to extend continuous selections defined on large compact
subsets of $[0,T]$ to global continuous selections of $C(\cdot)$ while
preserving uniform bounds.
After introducing admissible trajectories, we recall the standard
differential-measure formulation and then propose an integral formulation
adapted to uniformly prox-regular moving sets.
We also introduce a variational residual $\mathcal E_\nu$ that measures the
defect in the integral inequality and provides a bridge between the two
formulations.
Our main result shows that, under \ref{H1}, \ref{H2}, and \ref{H3}, the two
definitions characterize the same class of trajectories.

\begin{definition} Let $C\colon [0,T]\rightrightarrows \H$ be a set-valued map. We say that a mapping $x\colon [0,T]\to \H$ is an \emph{admissible trajectory for $C$} if $x(\cdot)$ is right-continuous and of bounded variation, with $x(0)=x_0\in C(0)$ and $x(t)\in C(t)$ for all $t\in[0,T]$. 
\end{definition}

We now recall the notion of solution in the sense of differential measures for
the sweeping process, following \cite{MR3571564}, which is the standard
formulation in the literature.
	\begin{definition}\label{def2}
		We say that $x\colon [0,T]\to \H$ is a \emph{solution in the sense of differential measures} of the sweeping process if  
		\begin{enumerate}
			\item[(a)] The mapping $x(\cdot)$ is an admissible trajectory for $C$.
			\item[(b)]  There exists a complete positive Radon measure $\nu$ such that the differential measure $dx$ of $x(\cdot)$ is absolutely continuous with respect to $\nu$, with density $\frac{dx}{\dnu}(\cdot)\in L_{\nu}^1([0,T],\H)$, and  
			$$
			\frac{dx}{\dnu}(t)\in -N^P(C(t);x(t)) \quad \nu\textrm{-a.e. } t\in [0,T].
			$$
		\end{enumerate}
	\end{definition}

Now we introduce the notion of an \emph{integral solution}, following \cite{MR2774131,MR2491851} (see also \cite{MR3369277} for the convex case). To the best of our knowledge, the corresponding definition for prox-regular sets is new.
	\begin{definition}\label{def1}
		We say that $x\colon [0,T]\to \H$ is an \emph{integral solution} of the sweeping process if
		\begin{enumerate}
			\item[(a)] The mapping $x(\cdot)$ is an admissible trajectory for $C$.
			\item[(b)] There exists a complete positive Radon measure $\nu$ such that the differential measure $dx$ of $x(\cdot)$ is absolutely continuous with respect to $\nu$, with density $v(\cdot):=\frac{dx}{\dnu}(\cdot) \in L_{\nu}^1([0,T],\H)$, and such that, for every $y\in \mathcal{C}([0,T],\H)$ satisfying   $y(t)\in C(t)$ for $\nu$-a.e.\ $t\in [0,T]$, one has
			\begin{equation}\label{eqn.ineq.1}
				\int_0^T \left(\langle v(t),y(t)-x(t) \rangle + \frac{\Vert v(t)\Vert }{2\rho}\,\|y(t)-x(t)\|^2 \right)\, d\nu(t)\geq 0.
			\end{equation}
		\end{enumerate}
	\end{definition}

	\begin{lemma}\label{lemma-measurable-fx-mx} 
    Assume that \ref{H1} and \ref{H2} hold.     Let $x(\cdot)$ be a mapping of bounded variation, and assume that its differential measure $dx$ is absolutely continuous with respect to a complete positive Radon measure $\nu$ on $[0,T]$. Define  
    $$
    v(t):=\frac{dx}{d\nu}(t) \,\textrm{ and }\, f_x(t,y):= \delta_{C(t)}(y) + \left\langle v(t),\,y-x(t)\right\rangle + \frac{\Vert v(t)\Vert }{2\rho}\,\|y-x(t)\|^2.
    $$
    Then, for every continuous mapping $y\colon [0,T]\to \H$, the functions 
        $$
        t\mapsto f_x(t,y(t)) \quad \textrm{ and } \quad t\mapsto m_{f_x}(t) := \inf_{y\in \H} f_x(t,y)
        $$
        are $\nu$-measurable.
	\end{lemma}
	\begin{proof} 
    It remains to prove the measurability of the map $t\mapsto \delta_{C(t)}(y(t))$.  Set 
    $$
    \psi(t,z):=\delta_{C(t)}(z) \textrm{ for  } (t,z)\in [0,T]\times \H.
    $$
    For every $t\in [0,T]$, the function $\psi(t,\cdot)$ is lower semicontinuous and $\operatorname{epi}\psi_t=C(t)\times [0,+\infty[$, where $\psi_t:=\psi(t,\cdot)$. Since $C$ is lower semicontinuous, the set-valued mapping $t\mapsto \operatorname{epi} \psi_t$ is lower semicontinuous. hence, by \cite[Lemma 3.1]{MR4547969}, the map $t\mapsto \delta_{C(t)}(y(t))$ is measurable. Therefore, $t\mapsto f_x(t,y(t))$ is $\nu$-measurable.  Next, define 
    $$
    g\colon (t,z)\mapsto \left\langle v(t),z-x(t)\right\rangle + \frac{\Vert v(t)\Vert }{2\rho}\,\|z-x(t)\|^2.
    $$
 By Lusin's Theorem, there exists a sequence of compact sets $(K_n)$ such that $\nu([0,T]\setminus K_n)<\frac{1}{n}$ and both $v(t)|_{K_n}$ and $x|_{K_n}$ are continuous. Fix $\alpha\in\R$, $n\in \mathbb{N}$,  and $\bar{t}\in \left\{t\in K_n : m_{f_x}(t)<\alpha\right\}$. Then there exists $\bar y\in C(\bar t)$ such that $g(\bar t,\bar y)<\alpha$. Define $\mathscr{C}(t) = C(t)$ for all $t\neq \bar t$ and $\mathscr{C}(\bar t) = \{\bar y\}$. This set-valued mapping is lower semicontinuous (by Lemma \ref{patching-lemma}), takes $\rho$-uniformly prox-regular values and path-connected values (by Proposition \ref{path-connected-prox-reg}), and therefore admits a continuous selection $y\colon [0,T]\to \H$ by \cite[Theorem 5]{MR2406391}. Since $t\mapsto g(t,y(t))$ is continuous on $K_n$, there exists an open neighborhood $U$ of $\bar t$ such that $g(t,y(t))<\alpha$ for all $t\in U\cap K_n$. As $y(t)\in C(t)$ for all $t\in [0,T]$, we have $m_{f_x}(t)\leq f_x(t,y(t))=g(t,y(t))<\alpha$ on $U\cap K_n$. Thus $\left\{t\in K_n : m_{f_x}(t)<\alpha\right\}$ is relatively open in $K_n$, hence measurable. It follows that $\left\{t\in \bigcup_{n\in\N} K_n : m_{f_x}(t)<\alpha\right\}$ is measurable. Since  $[0,T]\setminus \bigcup_{n\in\N}K_n$ is $\nu$-null and $\nu$ is complete, $m_f$ is a $\nu$-measurable.
\end{proof}

Let $x(\cdot)$ be a mapping of bounded variation, and assume that its differential measure $dx$ is absolutely continuous with respect to a complete positive Radon measure $\nu$ on $[0,T]$. We define the \emph{variational residual} by
\begin{equation}\label{eq:def_E_nu}
    \mathcal{E}_{\nu}(x)=\inf_{y\in \mathcal{A}_{\nu}}\int_0^T \left[ \left\langle v(t),\,y(t)-x(t)\right\rangle + \frac{\Vert v(t)\Vert }{2\rho}\,\|y(t)-x(t)\|^2\right] \,d\nu(t),
\end{equation}
where $v:=\frac{dx}{d\nu}$ and \(\mathcal{A}_{\nu}\) denotes the class of \emph{admissible test trajectories},
\[
\mathcal{A}_{\nu}:=\Bigl\{\,y\in \mathcal{C}([0,T];\H): y(t)\in C(t)\ \textrm{ for }\nu\textrm{-a.e.\ }t\in[0,T]\,\Bigr\}.
\]
The variational residual $\mathcal E_\nu$ is designed so that integral solutions satisfy $\mathcal E_\nu(x)\ge 0$, since the integral in \eqref{eq:def_E_nu} is then nonnegative for every admissible test trajectory. 
In order to relate the condition $\mathcal E_\nu(x)=0$ to the integral formulation, one also needs a converse normalization: for any admissible trajectory $x$, the class $\mathcal A_{\nu}$ should be nonempty and contain test trajectories along which the value of the integral in \eqref{eq:def_E_nu} can be made arbitrarily small. 
The next proposition establishes precisely these two facts. Its proof uses \ref{H3} to extend continuous selections defined on large compact subsets of $[0,T]$ to global continuous selections of $C$.

\begin{proposition}\label{A-negative}
Assume that \ref{H1}, \ref{H2}, and \ref{H3} hold. Let $x(\cdot)$ be an admissible trajectory for $C$, and assume that the differential measure $dx$ is absolutely continuous with respect to a complete positive Radon measure $\nu$ on $[0,T]$. 
Then $\mathcal A_{\nu}\neq\emptyset$ and $\mathcal{E}_{\nu}(x)\leq 0$.
\end{proposition}
\begin{proof} Let us denote $v:=\frac{dx}{d\nu}\in L^1_\nu([0,T],\H)$. Since $C$ is lower semicontinuous and takes uniformly prox-regular and connected values, Proposition \ref{path-connected-prox-reg} implies that $C$ takes path-connected values. Hence, by \cite[Theorem 5]{MR2406391}, $C$ admits a continuous selection, and therefore  $\mathcal{A}_{\nu}\neq\emptyset$. Fix $\eta>0$. For each $n\in \N$, by Lusin's theorem there exists a compact set  $K_n\subset [0,T]$ such that $\nu([0,T]\setminus K_n)\leq\frac{1}{n}$ and the restriction $x|_{K_n}\colon  K_n\to \H$ is continuous. By \ref{H3}, there exists $R_{x,\eta}>0$ and a continuous selection $y_n\colon [0,T]\to \H$ of $C$ such that $y_n = x$ on $K_n$ and $\sup_{t\in [0,T]}\|y_n(t)\|\leq R_{x,\eta}$. It follows that, for all $n\in \N$,
\begin{equation*}
    \begin{aligned}
    \mathcal{E}_{\nu}(x) &\leq\int_0^T \left[ \left\langle v(t),\,y_n(t)-x(t)\right\rangle + \frac{\Vert v(t)\Vert }{2\rho}\,\|y_n(t)-x(t)\|^2\right] \,d\nu(t)\\
    &= \int_{[0,T]\setminus K_n} \left[ \left\langle v(t),\,y_n(t)-x(t)\right\rangle + \frac{\Vert v(t)\Vert }{2\rho}\,\|y_n(t)-x(t)\|^2\right] \,d\nu(t)\\
    &\leq \left(1+\frac{1}{2\rho}\right)(R_{x,\eta} + \mathscr{C} + (R_{x,\eta} + \mathscr{C})^2)\int_{[0,T]\setminus K_n} \|v(t)\|\,d\nu(t),
    \end{aligned}
\end{equation*}
where $\mathscr{C} := \sup_{t\in [0,T]}\|x(t)\|$.  Since $\nu([0,T]\setminus K_n)\to 0$ and $v\in L^1_\nu([0,T];\H)$, we have $$\lim_{n\to \infty}\int_{[0,T]\setminus K_n} \|v(t)\|d\nu(t)= 0.$$ Therefore $\mathcal{E}_{\nu}(x)\leq 0$.
\end{proof}

The next result establishes the equivalence between Definitions~\ref{def2} and~\ref{def1}, which was initially proved for set-valued maps with nonempty, closed, convex, and bounded values in \cite[Theorem~9.3]{MR4547969}.
	\begin{theorem}\label{thm_main}
		Assume \ref{H1}, \ref{H2} and \ref{H3} hold. Let $x\colon [0,T]\to \H$ be a mapping.  Then $x(\cdot)$ is an integral solution of sweeping process if and only if $x(\cdot)$ is a solution in the sense of differential measures of the sweeping process.
	\end{theorem}
	\begin{proof}
    First of all, if $x$ is a solution in the sense of differential measures, it is straightforward to conclude that it is also an integral solution.\\
    Conversely, suppose that $x$ is an integral solution and denote $v(t):=\frac{dx}{d\nu}(t)$. Consider the functions $f_x$ and $m_{f_x}$ defined previously in Lemma \ref{lemma-measurable-fx-mx}. Observe that \eqref{eqn.ineq.1} can be rewritten as
		\begin{equation}\label{inf-cont-pr-reg}
			\inf_{y\in \mathcal{C}([0,T];\H)}\int_0^T f_x(t,y(t))\,d\nu(t)\geq 0. 
		\end{equation}
		Therefore, $\mathcal{E}_{\nu}(x)\geq 0$ and by Proposition \ref{A-negative} we have $\mathcal{E}_{\nu}(x) = 0$. Our goal is to prove that $\nu$-a.e. $m_{f_x}(t) = 0$. Since $x(\cdot)$ is admissible,  we have that $m_{f_x}(t)\leq 0$ for all $t\in [0,T]$.\\
        \emph{{\bf Claim 1:} The function $m_f$ is integrable.}\\
		\emph{Proof of Claim 1:} Indeed, if $v(t) = 0$, then $m_f(t) = 0$. Assume that $v(t) \neq 0$, then for $y\in C(t)$
        \begin{equation*}
            \begin{aligned}
                \left\langle v(t),\,y-x(t)\right\rangle + \frac{\Vert v(t)\Vert }{2\rho}\|y-x(t)\|^2
                &=  \frac{\Vert v(t)\Vert}{2\rho}(\|y-x(t)+\rho \mathsf{w}\|^2-\rho^2)\\
                &\geq -\frac{\rho}{2}\Vert v(t)\Vert,
            \end{aligned}
        \end{equation*}
        where $\mathsf{w} := \frac{v(t)}{\|v(t)\|}$. Therefore, for all $t\in [0,T]$, one has  $-\tfrac{\rho}{2}\left\|v(t)\right\|\leq m_f(t)\leq 0$,  which means that $m_f\in L^1_{\nu}([0,T])$, concluding the claim. \qed \\
		Define the sets
        $$
        A_1 := \{t\in [0,T] : v(t)\neq 0\}  \textrm{ and } A_2 = \{t\in A_1 : m_{f_x}(t)<0\}.
        $$
      Consider $n\in \N$, and the following function $\beta_n := (1-\frac{1}{n})m_{f_x}$ on $A_2$ and $\beta_n = 0$ on $[0,T]\setminus A_2$. Note that $m_{f_x}<\beta_n<0$ on $A_2$ and $$\int_0^T \beta_n(t)d\nu(t)\leq \int_0^Tm_{f_x}(t)d\nu(t) + \frac{\|m_{f_x}\|_1}{n}.$$
		By Lusin's Theorem, for every $n\in \N$, we can find a compact set  $K_n\subset [0,T]$ such that $\beta_n|_{K_n}$, $x|_{K_n}$ and $v|_{K_n}$ are continuous where $\nu([0,T]\setminus K_n)<\frac{1}{n}$. By regularity of $\nu$, we can take compacts set $K_n^1\subset K_n\setminus A_2$ and $K_n^2\subset A_2\cap K_n$ such that $\nu(K_n\setminus A_2\setminus K_n^1)<\frac{1}{n}$ and $\nu(A_2\cap K_n\setminus K_n^2)<\frac{1}{n}$.
		We consider the set-valued mapping $t \rightrightarrows  \mathsf{M}_n(t)$ given by
		\begin{equation}\label{eqn.Mn}
			\mathsf{M}_n(t) = \{y\in \H : f_x(t,y)\leq \beta_n(t)\}.
		\end{equation}
        Note that $\mathsf{M}_n(t)\neq \emptyset$ for all $t\in [0,T]$ by definition of $\beta_n$.\\
		\emph{{\bf Claim 2:} We have 
			\[
			\mathsf{M}_n(t) =
			\begin{cases} 
				C(t) & \textrm{if } t\in [0,T]\setminus A_1, \\
				C(t)\cap \mathbb{B}_{r(t)}[a(t)] & \textrm{if } t\in A_1.
			\end{cases}
			\]
			where $a(t) = x(t) - \rho\frac{ v (t)}{\| v (t)\|}$ and $r(t) = \sqrt{\rho^2 + 2\rho\frac{\beta_n(t)}{\| v (t)\|}}$. Moreover, for all $t\in A_2$, $C(t)\cap \mathbb{B}_{r(t)}(a(t))\neq \emptyset$.}\\
		\emph{Proof of Claim 2:} If $t\in [0,T]\setminus A_1$ we have $ v (t) = 0$. We also note that $\beta_n = 0$ on $[0,T]\setminus A_2$ and since $A_2\subset A_1$, it follows $\beta_n(t) = 0$. Then
		\begin{equation*}
			\{y\in \H : f_x(t,y)\leq \beta_n(t)\} = \{y\in\H : \delta_{C(t)}(y)\leq 0\} = C(t).
		\end{equation*}
		If $t\in A_1$, we have $ v (t)\neq 0$. Pick $y\in \mathsf{M}_n(t)$, it equivales to $y\in C(t)$ (since $\beta_n(t)<\infty$) and 
		\begin{equation}\label{eqn-ball3432}
			2\left\langle \rho\mathsf{w},y-x(t)\right\rangle + \|y-x(t)\|^2\leq \frac{2\rho\beta_n(t)}{\| v (t)\|}
		\end{equation}
		where $\mathsf{w} := \frac{ v (t)}{\| v (t)\|}$, and finally observe that \eqref{eqn-ball3432} can be written as $\|y-x(t) + \rho \mathsf{w}\|\leq \sqrt{\frac{2\rho\beta_n(t)}{\| v (t)\|} + \rho^2}$, and we are done. Finally, by definition of $\beta_n$, for all $t\in A_2$ there is $y\in \H$ such that $f_x(t,y)<\beta_n(t)$ and as before this inequality is equivalent to $y\in C(t)\cap \mathbb{B}_{r(t)}(a(t))$ and we have proved the claim.\\
		\emph{{\bf Claim 3:} $\mathsf{M}_n$ has a continuous selection on $\mathcal{K}_n := K_n^1\cup K_n^2$.}\\
		\emph{Proof of Claim 3:} Note that for $t\in A_2$, $r(t)<\rho$. From Lemma \ref{lemma-prox-inter-ball}, we have $\mathsf{M}_n(t)$ is $\rho$-uniformly prox-regular and nonempty for all $t\in A_2\cup [0,T]\setminus A_1$. On $K_n^1$, we have a continuous selection of $\mathsf{M}_n$ given by $t\mapsto x(t)$. So, now we are looking for a continuous selection on $K_n^2\subset A_2\cap K_n$. Define $\mathscr{M}_n\colon t\in K_n^2 \rightrightarrows  C(t)\cap \mathbb{B}_{r(t)}(a(t))$ and note that it satisfies the hypothesis of \cite[Proposition 5, p. 44]{MR755330}, then it is lsc. By Lemma \ref{lemma-prox-inter-ball} we have for all $t\in K_n^2$, $\mathsf{M}_n(t) = \cl{\mathscr{M}_n(t)}$, then it is still lsc and by using \cite[Corollary 15.137]{MR4659163}, we obtain a continuous selection of $\mathsf{M}_n$ on $K_n^2$ denoted by $\tilde{x}_n\colon K_n^2\to \H$. Then, we have that $y_n(t) = x\mathds{1}_{K_n^1} + \tilde{x}_n\mathds{1}_{K_n^2}$ is continuous on $\mathcal{K}_n$ and we are done.\\
		\emph{{\bf Claim 4:} Theorem \ref{thm_main} holds.}\\
		\emph{Proof of Claim 4:} Note that $\sup_{t\in \mathcal{K}_n}\|y_n(t)-x(t)\|\leq 2\rho$. By using \ref{H3} with the compact $\mathcal{K}_n$ and $\eta:=2\rho$, there is a continuous selection $\bar y_n\colon [0,T]\to \H$ of $C$ such that $\bar{y}_n= y_n$ on $\mathcal{K}_n$ and $\sup_{t\in [0,T]}\|\bar y_n(t)\|\leq R_{x,\eta}$. 
		By \eqref{eqn.Mn}, we have $$\int_{\mathcal{K}_n}f_x(t,y_n(t))d\nu(t)\leq \int_{\mathcal{K}_n}\beta_n(t)d\nu(t)\leq \int_{\mathcal{K}_n} m_{f_x}(t)d\nu(t) + \frac{\|m_{f_x}\|_1}{n}.$$
		Observe that
		\begin{align*}
			\int_{\mathcal{K}_n} f_x(t,y_n(t))d\nu(t) =& \int_0^T f_x(t,\bar{y}_n(t))d\nu(t)-\int_{[0,T]\setminus\mathcal{K}_n}f_x(t,\bar{y}_n(t))d\nu(t) \\
            &+ \int_{\mathcal{K}_n} f_x(t,y_n(t)) - f_x(t,\bar{y}_n(t))d\nu(t)\\
			\geq&-\int_{[0,T]\setminus\mathcal{K}_n}f_x(t,\bar{y}_n(t))d\nu(t),
		\end{align*}
		where in the third line we have used \eqref{inf-cont-pr-reg}. Define $\mathscr{C} := \sup_{t\in [0,T]}\|x(t)\|$. Note that
		\begin{align*}
			&\left|\int_{[0,T]\setminus \mathcal{K}_n} f_x(t,\bar{y}_n(t))d\nu(t) \right|\\ \leq &\int_{[0,T]\setminus \mathcal{K}_n} \left|\left\langle  v (t),\bar{y}_n(t)-x(t)\right\rangle\right|d\nu(t)+ \int_{[0,T]\setminus \mathcal{K}_n}\frac{\|v(t)\|}{2\rho}\|\bar{y}_n(t) - x(t)\|^2d\nu(t)\\
			\leq & \left(1+\frac{1}{2\rho}\right)(R_{x,\eta}+\mathscr{C} + (R_{x,\eta}+\mathscr{C})^2)\int_{[0,T]\setminus \mathcal{K}_n} \left\| v (t)\right\|d\nu(t).
		\end{align*}
		We also know that $\nu([0,T]\setminus \mathcal{K}_n)<\frac{3}{n}$. It follows that
		\begin{align*}
			&\frac{\|m_{f_x}\|_1}{n} + \int_{\mathcal{K}_n}m_{f_x}(t)d\nu(t) \\\geq & -\left(1+\frac{1}{2\rho}\right)(R_{x,\eta}+\mathscr{C} + (R_{x,\eta}+\mathscr{C})^2)\int_{[0,T]\setminus \mathcal{K}_n} \left\| v (t)\right\|d\nu(t).
		\end{align*}
		By using that $m_{f_x}(\cdot)$ and $ v(\cdot)$ are integrable, taking $n\to \infty$ in the last inequality, we have $\int_0^T m_{f_x}(t)d\nu(t)\geq 0$, and since $m_{f_x}\leq 0$ $\nu$-a.e., we conclude $m_{f_x}\geq 0$ $\nu$-a.e. which implies that $\nu$-a.e. $$\forall y\in C(t) : \left\langle - v (t),y-x(t) \right\rangle\leq \frac{\| v (t)\|}{2\rho}\|y-x(t)\|^2,$$ it follows that $- v (t)\in N^P(C(t);x(t))$ $\nu$-a.e. which means that $x$ is a solution  in the sense of differential measures. 
	\end{proof}

The following remark discusses the sharpness of assumption \ref{H3} in Theorem~(\ref{thm_main}).
\begin{remark}
Example~\ref{ex:H3-fails} shows that \ref{H3} is a quantitative, not a
qualitative, assumption. Two observations put its role in perspective.
First, an inspection of the proofs of Proposition~\ref{A-negative} and of
Claim~4 in Theorem~\ref{thm_main} shows that the bound $R_{x,\eta}$ is
used only through the tail products
\[
\Bigl(1+\dfrac{1}{2\rho}\Bigr)
\bigl(R_{x,\eta}+\mathscr C+(R_{x,\eta}+\mathscr C)^2\bigr)
\int_{[0,T]\setminus \mathcal K_n}\|v(t)\|\,d\nu(t).
\]
Consequently, both results remain valid if \ref{H3} is replaced by the
weaker, trajectory-dependent condition: for every admissible trajectory
$x$ with $dx=v\,d\nu$ and every sequence of compact sets
$K_n\subset[0,T]$ with $\nu([0,T]\setminus K_n)\to0$ on which the
relevant partial selections are defined, there exist continuous
extensions $\bar y_n$ such that
\(
\bigl(1+\|\bar y_n\|_\infty\bigr)^2
\int_{[0,T]\setminus K_n}\|v\|\,d\nu\to0.
\)
In Example~\ref{ex:H3-fails}, this permits extension bounds growing like
$1/\delta_n$, provided the tail of $\|v\|\,d\nu$ near $t=\tfrac12$ decays
sufficiently fast. Second, whether the equivalence of
Theorem~\ref{thm_main} can actually fail in the absence of \ref{H3}
appears to be a delicate open question. Indeed, completing the square in
the integrand gives, for every $y\in C(t)$,
\[
\langle v(t),y-x(t)\rangle+\frac{\|v(t)\|}{2\rho}\|y-x(t)\|^2
\ \ge\ \|v(t)\|\,\|y-x(t)\|\left(\frac{\|y-x(t)\|}{2\rho}-1\right),
\]
so that any competitor producing a negative value lies in the open
$2\rho$-tube around $x(t)$, while excursions of a test trajectory outside
this tube contribute nonnegatively wherever $\|v\|\,d\nu$ carries mass.
A counterexample would therefore require pointwise violations detectable
only through long excursions performed during time windows charged by
$\|v\|\,d\nu$, a mechanism that uniform prox-regularity itself tends to
preclude. We leave the sharpness of \ref{H3} as an open problem.
\end{remark}

\section{Br\'ezis-Ekeland-Nayroles-Type Principle for Sweeping Processes}\label{section5}

The Br\'ezis-Ekeland variational principle and Nayroles' minimum theorem yield a global variational characterization for a wide class of evolution problems generated by convex potentials or maximal monotone operators: one assigns to each trajectory a nonnegative functional whose minimum value is \(0\), and solutions are exactly the trajectories attaining this minimum \cite{MR637215,MR637214,MR418609}. This approach is well suited to existence by direct methods and to stability and approximation analyses, including time-discretization \cite{MR2425653,MR2531193,MR2489147}.

In this section we establish an analogous characterization for sweeping processes driven by uniformly prox-regular moving sets. We introduce a prox-regular variational residual and prove that, under our standing assumptions, its vanishing is equivalent to both the integral and the differential-measure formulations. This yields a global variational criterion for solutions in a nonconvex geometric setting and highlights the connection between catching-up compactness arguments and variational techniques.
\begin{theorem}\label{thm_BEN}
Assume that \ref{H1}, \ref{H2}, and \ref{H3} hold. Let $x\colon [0,T]\to \H$ be an admissible trajectory for $C$.  Assume that the differential measure $dx$ is absolutely continuous with respect to a complete positive Radon measure $\nu$ on $[0,T]$. Let $\mathcal E_\nu(x)$ denote the variational residual defined in \eqref{eq:def_E_nu}.  Then the following assertions are equivalent:
\begin{enumerate}
\item[(i)] The trajectory $x$ is an integral solution of the sweeping process.
\item[(ii)] The trajectory $x$ is a solution in the sense of differential measures of the sweeping process.
\item[(iii)] The differential measure $dx$ is absolutely continuous with respect to a complete positive Radon measure $\nu$ on $[0,T]$ and $\mathcal{E}_{\nu}(x)=0$.
\end{enumerate}
\end{theorem}
\begin{proof} The equivalence between \emph{(i)} and \emph{(ii)} follows from Theorem~\ref{thm_main}.
Moreover, Proposition~\ref{A-negative} ensures that $\mathcal A_{\nu}\neq\emptyset$ and that $\mathcal E_\nu(x)\le 0$. Denote $v:=\frac{dx}{d\nu}\in L^1_\nu([0,T];\H)$.\\
\noindent\emph{(i)$\Rightarrow$(iii).} Let $x$ be an integral solution. Then, by Definition~\ref{def1}, for every $y\in\mathcal A_{\nu}$,
\[
\int_0^T \Big[ \langle v(t),y(t)-x(t)\rangle + \tfrac{\|v(t)\|}{2\rho}\|y(t)-x(t)\|^2\Big]\,d\nu(t)\ge 0.
\]
Taking the infimum over $y\in\mathcal A_{\nu}$ yields $\mathcal E_\nu(x)\ge 0$. Combined with
$\mathcal E_\nu(x)\le 0$ from Proposition~\ref{A-negative}, we obtain $\mathcal E_\nu(x)=0$.\\
\noindent\emph{(iii)$\Rightarrow$(i).} Assume that $\mathcal E_\nu(x)=0$. Then, by definition, for every $y\in\mathcal A_{\nu}$,
\begin{equation*}
\begin{aligned}
\int_0^T \Big[ &\langle v(t),y(t)-x(t)\rangle + \tfrac{\|v(t)\|}{2\rho}\|y(t)-x(t)\|^2\Big]\,d\nu(t)\\
&\quad \quad \quad \, \ge \inf_{z\in\mathcal A_{\nu}}\int_0^T \Big[ \langle v(t),z(t)-x(t)\rangle + \tfrac{\|v(t)\|}{2\rho}\|z(t)-x(t)\|^2\Big]\,d\nu(t)
=0.
\end{aligned}
\end{equation*}
Hence $x$ is an integral solution. This completes the proof.
\end{proof}

\section{Stability of Approximations via the Prox-Regular Variational Residual}\label{section6}

As an application of the Br\'ezis-Ekeland-Nayroles-type principle, we establish a residual-based stability result for sweeping processes
under perturbations of the moving set.
We consider a target set-valued map $C\colon[0,T] \rightrightarrows \H$ together with a
sequence of approximations $C_n\colon[0,T] \rightrightarrows \H$ (for instance, arising
from time discretization or geometric regularization).
Our aim is to identify conditions under which any sequence of admissible
trajectories for $C_n$ converges, as $n\to\infty$, to a solution of the limit
sweeping process driven by $C$.

The main tool is the prox-regular variational residual. It assigns to a
trajectory $x(\cdot)$ a scalar quantity quantifying the defect in the
integral inequality that characterizes solutions. In particular, the
residual is \emph{nonnegative} on solutions, and under our standing
assumptions it vanishes exactly at the solutions. In the approximation
setting, one does not expect the residual of a discrete or perturbed
trajectory to be zero. Instead, one requires an \emph{asymptotic
zero-residual} property, namely that the residual of the approximations
tends to zero. This provides a robust way to identify the limit without
passing directly to the normal-cone inclusion.

Given set-valued maps $C_n\colon [0,T] \rightrightarrows  \H$ and a complete positive Radon measure $\nu$ on $[0,T]$, we define the class of \emph{admissible test trajectories} for $C_n$ by
\[
\mathcal{A}_{\nu}^{n}:=\Bigl\{\,y\in \mathcal{C}([0,T];\H): y(t)\in C_n(t)\ \textrm{for }\nu\textrm{-a.e.\ }t\in[0,T]\,\Bigr\},
\]
These test trajectories play the same role as the admissible paths
$\mathcal A_{\nu}$ for the limit map $C$ in Definition~\ref{def1}.

To guarantee that limits of feasible points for $C_n$ remain feasible for
$C$, we assume the following \emph{Outer Semicontinuity Condition}:
\begin{quote}
\textrm{(OSC)} For every $t\in [0,T]$, if $z_n\in C_n(t)$ and $z_n\to z$ in
$\H$, then $z\in C(t)$.
\end{quote}

In addition, to pass to the limit in the residual inequalities, we need
to approximate continuous selections of $C$ by continuous selections of
$C_n$. This is encoded by the \emph{Selection Approximation Property}:
\begin{equation}\label{SAP}\tag{SAP}
\forall\,y\in\mathcal A_{\nu}\ \exists\,(y_n)_{n\in\mathbb N}\ \textrm{ such that }\ y_n\in\mathcal A_{\nu}^n\ \textrm{ and }\ 
\|y_n-y\|_{\infty}\xrightarrow[n\to\infty]{}0,
\end{equation}
where $\|z\|_\infty:=\sup_{t\in[0,T]}\|z(t)\|$.
Condition~\eqref{SAP} ensures that any fixed test trajectory for the
limit problem can be approximated by tests for the perturbed problems,
which is essential to transfer variational inequalities from $C_n$ to
$C$.

Let $x_n(\cdot)$ be of bounded variation and assume that $dx_n\ll\nu$,
with density $v_n:=\frac{dx_n}{d\nu}$. We define the prox-regular
variational residual associated with $C_n$ by
\begin{equation*}
\mathcal{E}_{\nu}^n(x_n)
:=\inf_{y\in \mathcal{A}_{\nu}^{n}}
\int_0^T\Big[
\langle v_n(t),y(t)-x_n(t)\rangle
+\frac{\|v_n(t)\|}{2\rho}\|y(t)-x_n(t)\|^2
\Big]\,d\nu(t).
\end{equation*}
By the Br\'ezis-Ekeland-Nayroles-type principle proved above, under our standing
assumptions the identity $\mathcal E_\nu^n(x_n)=0$ is equivalent to
$x_n(\cdot)$ being a solution of the sweeping process driven by $C_n$
(in the integral sense, equivalently in the sense of differential
measures). Thus, the convergence $\mathcal E_\nu^n(x_n)\to 0$ is an
\emph{asymptotic consistency} requirement for the approximations.

We now state a stability result: a uniform limit of admissible
trajectories with vanishing residual for $C_n$ is a solution of the limit
sweeping process driven by $C$.

\begin{theorem}\label{thm-conv-last}
Assume that $C_n, C\colon [0,T] \rightrightarrows  \H$ satisfy \ref{H1}, \ref{H2}, and \ref{H3} with the same constant $\rho$, and that \eqref{SAP} and (OSC) hold. Fix a complete positive Radon measure $\nu$ on $[0,T]$. 
For each $n$, let $x_n$ be an admissible trajectory for $C_n$ with $x_n(0)=x_0$, $dx_n\ll \nu$, and define $v_n:=\frac{dx_n}{d\nu}$. Assume that $x_n\to x$ uniformly on $[0,T]$ for some mapping $x\colon[0,T]\to\H$, and that
\begin{enumerate}[label=\textnormal{(\alph*)}]
    \item The sequence $(v_n)$ is bounded in $L^p_\nu([0,T];\H)$ for some $p\in ]1,\infty]$ and $|dx_n| \rightharpoonup^{\ast} |dx|$ in $\mathcal{M}([0,T])$, where $dx$ denotes the differential measure of  \(x\).
    \item \(\mathcal{E}_{\nu}^n(x_n)\to 0\).
\end{enumerate}
Then \(x(\cdot)\) is an integral solution. Equivalently, \(x(\cdot)\) is a solution in the sense of differential measures.
\end{theorem}
\begin{proof} Fix $t\in [0,T]$. Since $x_n(t)\in C_n(t)$ for all $n\in \mathbb{N}$ and $x_n(t)\to x(t)$ by uniform convergence, assumption (OSC) implies $x(t)\in C(t)$. Moreover, $x(0)=\lim_n x_n(0)=x_0$. Since each $x_n$ is right-continuous and $x_n\to x$ uniformly, $x$ is right-continuous as well. \\
Consider $q\in [1,\infty[$ such that $\frac{1}{p}+\frac{1}{q} = 1$. Observe that for all $a,b\in [0,T]$, $\|x_n(b)-x_n(a)\| \leq \int_a^b\|v_n\|\,d\nu$, which implies that 
\begin{equation*}
    \operatorname{var}(x_n;[0,T])\leq \int_0^T\|v_n\|\,d\nu \leq \nu([0,T])^{1/q}\sup_{k\in \N}\|v_k\|_{L^p_\nu}.
\end{equation*}
From Helly's Theorem, up to a subsequence we have $dx_n\rightharpoonup^\ast dx$. \\
Fix $y\in \mathcal{A}_{\nu}$. By \eqref{SAP}, there exists $y_n\in \mathcal{A}_{\nu}^n$ such $\Vert y_n-y\Vert_{\infty}\to 0$. It follows that
\begin{equation}\label{eq:lowerbound}
\int_{0}^T \langle y_n-x_n,v_n\rangle d\nu
+\frac{1}{2\rho}\int_{0}^T \|v_n\|\cdot\|y_n-x_n\|^2\,d\nu
\ \ge\ \mathcal E_\nu^n(x_n).
\end{equation}
Since $(v_n)$ is uniformly bounded in $L^p_\nu([0,T];\H)$, passing to a subsequence, there is $v\in L^p_\nu([0,T];\H)$ such that $v_{n_k}\to v$ weakly in $v\in L^p_\nu([0,T];\H)$. Note that $y_n-x_n\to y-x$ uniformly, therefore $y_n-x_n\to y-x$ strongly in $L^q_\nu([0,T];\H)$, thus 
\begin{equation*}\label{eqn-limit2211}
    \lim_{k\to \infty}\int_0^T\langle y_{n_k}-x_{n_k},v_{n_k}\rangle d\nu = \int_0^T\langle y-x,v\rangle d\nu
\end{equation*} 
On the other hand, note that $f_n := \|y_n-x_n\|^2$ converges uniformly to $f:=\|y-x\|^2$. Since $(v_n)$ is bounded in $L^p_\nu([0,T];\H)$, by Hölder inequality we have
\begin{equation}\label{first-limit-key-last-result}
    \lim_{n\to \infty}\left|\int_0^T f_{n}\|v_{n}\|d\nu-\int_0^T f\|v_{n}\|d\nu\right| = 0
\end{equation}
Since $dx_n\rightharpoonup^{\ast} dx$ we have that for all continuous function $\varphi\colon [0,T]\to \H$
\begin{equation*}
    \int_0^T \langle \varphi,dx\rangle = \lim_{k\to \infty}\int_0^T \langle \varphi,dx_{n_k}\rangle = \lim_{k\to \infty} \int_0^T \langle \varphi,v_{n_k}\rangle d\nu = \int_0^T\langle \varphi,v\rangle d\nu
\end{equation*}
where we have used the weak convergence. The last fact implies that $dx\ll\nu$ and $v = \frac{dx}{d\nu}$. We are going to prove that
\begin{equation}\label{limit-keylastproof}
    \lim_{n\to \infty}\int_0^T f\|v_{n}\|d\nu = \int_0^T f\|v\|d\nu.
\end{equation}
Indeed, consider a sequence of continuous functions $(g_{m})$ such that $g_{m}\to f$ strongly in $L^q_\nu([0,T])$, thus for all $m,n\in \N$
\begin{equation*}
    \begin{aligned}
        \left| \int_0^T f\|v_{n}\|d\nu - \int_0^T f\|v\|d\nu \right|
        \leq & \  \left| \int_0^T f\|v_{n}\|d\nu - \int_0^T g_m\|v_{n}\|d\nu \right| + \left| \int_0^T g_m\|v_{n}\|d\nu - \int_0^T g_m\|v\|d\nu \right|\\
        &+ \ \left| \int_0^T g_m\|v\|d\nu - \int_0^T f\|v\|d\nu \right|\\
        \leq & \ \|f-g_m\|_{L^q_\nu}(\|v\|_{L^p_\nu}+\sup_{k\in \N}\|v_{k}\|_{L^p_\nu}) + \left| \int_0^T g_md|dx_{n}| - \int_0^T g_md|dx| \right| 
    \end{aligned}
\end{equation*}
where we have used Hölder inequality. Taking $n\to \infty$, by using that $|dx_n|\rightharpoonup^{\ast}|dx|$ we conclude that $\forall m\in \N$
\begin{equation*}
    \limsup_{n\to \infty} \left| \int_0^T f\|v_{n}\|d\nu - \int_0^T f\|v\|d\nu \right|\leq (\|v\|_{L^p_\nu} + \sup_{k\in \N}\|v_{k}\|_{L^p_\nu})\|f-g_m\|_{L^q_\nu}.
\end{equation*}
By sending $m\to \infty$, we get \eqref{limit-keylastproof}. From \eqref{first-limit-key-last-result} and \eqref{limit-keylastproof},  it follows that 
\begin{equation}\label{eqn-limit1122}
    \lim_{n\to \infty}\int_0^T f_n\|v_n\|d\nu = \int_0^T f\|v\|d\nu.
\end{equation}
Finally, taking limit in \eqref{eq:lowerbound} under the subsequence $(n_k)$, using \eqref{eqn-limit2211} and \eqref{eqn-limit1122} and assumption (b), we obtain for every $y\in \mathcal{A}_{\nu}$.
\begin{equation*}
\int_{0}^T \langle y-x,\,dx\rangle
+\frac{1}{2\rho}\int_{0}^T \|y-x\|^2\,d|dx|
\ \ge\ 0.
\end{equation*}
Therefore, we get
\begin{equation*}
\int_{0}^T \left[ \langle v(t),y(t)-x(t)\rangle
+\frac{\Vert v(t)\Vert }{2\rho}\|y(t)-x(t)\|^2\right]\,d\nu(t)\geq 0.
\end{equation*}
Therefore $x(\cdot)$ is an integral solution of the sweeping process driven by $C$. Finally, by Theorem~\ref{thm_BEN} (or Theorem~\ref{thm_main}), this is equivalent to $x(\cdot)$ being a solution
in the sense of differential measures.
\end{proof}
\begin{remark}
It is worth noting that, when \(C(t)\) is convex for every
\(t\in[0,T]\), the assumption $|dx_n|\rightharpoonup^\ast |dx|$ is not required in Theorem~\ref{thm-conv-last}.
\end{remark}

\section*{Concluding remarks}

We have investigated sweeping processes in a Hilbert space driven by time-dependent uniformly prox-regular sets, allowing the moving constraint to exhibit discontinuities of bounded variation. Our first contribution is the introduction of a global integral (variational) formulation adapted to the prox-regular setting and tested against \emph{continuous admissible functions}. In contrast with the convex case, the corresponding inequality necessarily contains a quadratic correction term, which compensates for the hypomonotonicity of proximal normal cones and restores a usable variational structure.

Under the standing assumptions \ref{H1}-\ref{H3}, we have shown that the
integral formulation is equivalent to the standard differential-measure
formulation.
A key ingredient is the bounded selection extension property \ref{H3}, which
ensures the availability of sufficiently rich continuous test trajectories
and permits the definition of a meaningful variational residual
$\mathcal E_\nu$.
This residual provides a quantitative measure of the defect in the integral
inequality and plays a central role in stability arguments.

Building on these equivalence results, we proved a
Br\'ezis-Ekeland-Nayroles-type variational principle for sweeping
processes with prox-regular moving sets.
This principle yields a global-in-time variational characterization of
solutions and leads to residual-based stability under perturbations of the
moving set, encompassing approximation procedures such as time
discretization and geometric regularization.

Several directions are suggested by these results.
A natural next step is to develop a discrete
Br\'ezis-Ekeland-Nayroles-type variational principle (in the spirit of
\cite{MR2531193}) and to exploit it in the analysis of inexact catching-up
schemes (in the line of \cite{MR4822735,MR4957011}), where residuals arise
intrinsically from numerical errors and inexact projections.
In this perspective, the residual $\mathcal E_\nu$ provides a principled
\emph{a posteriori} error indicator for approximation procedures and should
lead to quantitative stability and convergence estimates, as well as to
adaptive strategies driven by computable stopping criteria.
More broadly, it would be interesting to refine the residual-based stability
framework obtained here so as to capture finer modes of perturbation of the
moving set and to quantify the stability of solutions with respect to such
perturbations.

\section*{Acknowledgments}

We sincerely thank the reviewer for their insightful and constructive comments, which have significantly improved both the scope and the depth of the paper.

Juan Guillermo Garrido was supported by ANID Chile under grants CMM BASAL funds for Center of Excellence FB210005, Project ECOS230027, and ANID BECAS/DOCTORADO NACIONAL 21230802. Emilio Vilches was supported by ANID (Chile) through Fondecyt Regular grants No.~1240120, and No.~1261728, CMM BASAL funds for the Center of Excellence FB210005, and Project ECOS230027.

\nocite{*}
\bibliographystyle{abbrv}
\bibliography{references}

\end{document}